\documentclass[12pt,fullpage,doublespace]{article}
\usepackage{graphicx} 
\usepackage{setspace} 
\usepackage{latexsym} 
\usepackage{amsfonts} 
\usepackage{amsmath} 
\usepackage{amssymb}
\usepackage{accents}
\usepackage{textcomp}
\usepackage{undertilde}
\usepackage{enumerate}
\usepackage{bm}
\usepackage{stmaryrd}
\usepackage[margin=1in]{geometry}
\usepackage{amsthm}
\usepackage{ifsym}
\usepackage{amssymb,latexsym,amsmath}
\usepackage{graphics}
\usepackage{yfonts}
\usepackage{tikz}
\usetikzlibrary{matrix,arrows,positioning,scopes}

\newtheorem{theorem}{Theorem}[section]
\newtheorem{definition}[theorem]{Definition}
\newtheorem{proposition}[theorem]{Proposition}
\newtheorem{lemma}[theorem]{Lemma}
\newtheorem{corollary}[theorem]{Corollary}

\numberwithin{equation}{section}

\renewcommand{\gg}{\gamma}

\newcommand{\bR}{{\mathbb{R}}}

\newcommand{\rest}{\restriction}
\newcommand{\la}{\langle}
\newcommand{\ra}{\rangle}

\newcommand{\cp}{{\rm cp }}

\newcommand{\cf}{{\rm cf}}

\def\k{\kappa}
\def\a{\alpha}
\def\b{\beta}
\def\d{\delta}

\def\l{\lambda}

\newcommand{\card}[1]{{\vert #1 \vert} }

\renewcommand{\models}{\vDash}
\newcommand{\powerset}{{\cal P}}

\newcommand{\insegeq}{\trianglelefteq}
\newcommand{\inseg}{\triangleleft}

\def\P{{\mathcal{P} }}
\def\W{{\mathcal{W} }}
\def\Q{{\mathcal{ Q}}}

\def\K{{\mathcal{ K}}}

\def\R{{\mathcal R}}

\def\H{{\rm{HOD}}}
\def\M{{\mathcal{M}}}
\def\N{{\mathcal{N}}}

\def\T {{\mathcal{T}}}
\def\U{{\mathcal{U}}}
\def\S{{\mathcal{S}}}

\def\VT{{\vec{\mathcal{T}}}}

\newcommand{\rthm}[1]{Theorem~\ref{#1}}
\newcommand{\rlem}[1]{Lemma~\ref{#1}}

\newcommand{\rdef}[1]{Definition~\ref{#1}}

 \input xy
 \xyoption{all}
\begin{document}
\title{Non-tame mice from tame failures of the unique branch hypothesis}
\date{\today}
\author{Grigor Sargsyan\\
Department of Mathematics\\
Rutgers University,\\
New Brunswick, NJ, 08854 USA\\
http://math.rutgers.edu/$\sim$gs481\\
grigor@math.rutgers.edu\\
Nam Trang\\
Department of Mathematics\\
University of California\\
Berkeley, CA, 94720 USA\\
http://math.berkeley.edu/$\sim$namtrang\\
namtrang@math.berkeley.edu.}
\maketitle
\begin{abstract}
In this paper, we show that the failure of the unique branch hypothesis (UBH) for tame trees (see \rdef{tame iteration tree}) implies that in some homogenous generic extension of $V$ there is a transitive model $M$ containing $Ord \cup \mathbb{R}$ such that $M\vDash AD^+ + \Theta > \theta_0$. In particular, this implies the existence (in $V$) of a non-tame mouse. The results of this paper significantly extend Steel's earlier results from \cite{steel2002core} for tame trees. 
\end{abstract}
\thispagestyle{empty}

In this paper, we establish, using the core model induction, a lower bound for certain failures of the \textit{Unique Branch Hypothesis}, (UBH), which is the statement that every iteration tree that acts on $V$ has at most one cofinal well-founded branch. The following is our main theorem. Tame trees are defined in \rdef{tame iteration tree}: roughly speaking, these are the trees in which the critical point of any branch embedding is above a strong cardinal which reflects strong cardinals.

\begin{theorem}[Main Theorem]\label{main theorem} Suppose there is a proper class of strong cardinals and UBH fails for tame trees\footnote{For us, all iteration trees are non-overlapping}. Then in a set generic extension of $V$, there is a transitive inner model $M$ such that $Ord, \mathbb{R}\subseteq M$ and $M\models AD^++\theta_0<\Theta$. In particular, there is a non-tame mouse. 
\end{theorem}

UBH was first introduced by Martin and Steel in \cite{IT}. Towards showing UBH, Neeman, in \cite{Neeman}, showed that a certain weakening of UBH called cUBH holds provided there are no non-bland mice\footnote{We will not use this terminology.}. However, in \cite{woodin2010suitable}, Woodin showed that in the presence of supercompact cardinals UBH can fail for tame trees. It is, however, still an important open problem whether UBH holds for trees that use extenders that are $2^{\aleph_0}$-closed in the models that they are chosen from. A positive resolution of this problem will lead to the resolution of \textit{the inner model problem} for superstrong cardinals and beyond. It is worth remarking that the aforementioned form of UBH for tame trees will also lead to the resolution of the inner model problem for superstrong cardinals and beyond. Our work can be viewed as an attempt to prove UBH for tame trees by showing that its failure has a strong consistency strength.

In this direction, in \cite{steel2002core}, Steel showed that the failure of UBH for normal trees implies that there is an inner model with infinitely many Woodin cardinals. If in addition UBH fails for some tree $\T$ such that $\d(\T)$ is in the image of two branch embeddings witnessing the failure of UBH for $\T$ then Steel obtained an inner model with a strong cardinal which is a limit of Woodin cardinals. For tame trees (which include a class of examples constructed by Woodin in \cite{woodin2010suitable}), the Main Theorem considerably strengthens the aforementioned result of Steel and because the proof presented here is via the core model induction, we expect that it will yield much more: we believe that our proof, coupled with arguments from \cite{ATHM}, will give the existence of a transitive inner model $M$ such that $Ord, \mathbb{R}\subseteq M$ and $M\models ``AD_{\bR}+\Theta$ is regular". However, we still do not know if an arbitrary failure of UBH implies the existence of a non-tame mouse. Various arguments presented in this paper resemble the arguments given in \cite{StrengthPFA1} and \cite{PFA}, and some familiarity with  those articles will be useful. 

The first author's work was supported by NSF Grant No DMS-1201348. Part of this paper was written while the second author was visiting the first author who was a Leibniz Fellow at the Mathematisches Forschungsinstitut Oberwolfach. Finally, the authors would like to express their gratitude to their kind Rutgers colleagues, Lisa Carbone and Konstantin Mischaikow, for hosting both authors during the hurricane Sandy. The paper was finished while both authors were sheltered at their house.

\section{Preliminaries}

In this paper, we will need to make use of the material presented in Section 1 of \cite{StrengthPFA1},  most of which, especially Section 1.1, carries over to the hybrid context by just changing the word ``mouse" with ``hybrid mouse". Because of this, we will only introduce a few main notions and will use Section 1 of \cite{StrengthPFA1} as our main background material. In particular, we assume that the reader has already translated the material of Section 1.1 of \cite{StrengthPFA1} into the language of hybrid mice. 

\subsection{Stacking mice}

Following the notation of Section 1.3 of \cite{StrengthPFA1}, we fix some uncountable cardinal $\lambda$ and assume $ZF$. Notice that any function $f:H_{\l}\rightarrow H_{\l}$ can be naturally coded by a subset of $\powerset(\cup_{\k<\l}\powerset(\kappa))$. We then let $Code^*_\lambda:  H_{\lambda}^{H_{\l}}\rightarrow \powerset(\cup_{\k<\l}\powerset(\kappa))$ be one such coding. If $\l=\omega_1$ then we just write $Code^*$. Because for $\a\leq \l$, any $(\a, \lambda)$-iteration strategy for a hybrid premouse of size $<\l$ is in $H_{\lambda}^{H_{\l}}$, we have that any such strategy is in the domain of $Code^*_\l$. 

Suppose $\Lambda\in dom(Code^*_\l)$ is a strategy with hull condensation and $\mu\leq \l$. Recall that we say $F$ is $(\mu, \Lambda)$-mouse operator if for some $X\in H_\l$ and formula $\phi$ in the language of $\Lambda$-mice, whenever $Y$ is such that $X\in Y$, $F(Y)$ is the minimal $\mu$-iterable $\Lambda$-mouse satisfying $\phi[Y]$. We let $\M_F$ be such that $F$ is defined on all $Y$ containing $\M_F$. 

We then let $Code_\l$ be $Code^*_\l$ restricted to $F\in dom(Code^*_\l)$ which are defined by the following recursion. 
\begin{enumerate}
\item for some $\a\leq \l$, $F$ is a $(\a, \l)$-iteration strategy with hull condensation\footnote{In this case as well as in cases below $\a=0$ is allowed.},
\item for some $\a\leq \l$ and for some $(\a, \l)$-iteration strategy $\Lambda\in dom(Code^*_\l)$ with hull condensation, $F$ is a $(\l, \Lambda)$-mouse operator,
\item for some $\a\leq \l$, for some $(\a, \l)$-iteration strategy $\Lambda\in dom(Code^*_\l)$ with hull condensation, for some  $(\l, \Lambda)$-mouse operator $G\in dom(Code^*_\l)$ and for some $\b\leq \l$, $F$ is a $(\b, \Lambda)$-iteration strategy with hull condensation for some $G$-mouse $\M\in H_\l$.
\end{enumerate}
Given an $F\in Code_\l$ we let $\M_F$ be, in the case $F$ is an iteration strategy, the structure that $F$ iterates and, in the case $F$ is a mouse operator, the base of the cone on which $F$ is defined. 

Let $\P\in H_{\l}$ be a hybrid premouse and for some $\a\leq \l$, let $\Sigma$ be $(\a, \l)$-iteration strategy with hull condensation for $\P$. Suppose now that $\Gamma\subseteq \powerset(\cup_{\k<\l}\powerset(\kappa))$ is such that $Code_\l(\Sigma)\in \Gamma$. Given a $\Sigma$-premouse $\M$, we say $\M$ is $\Gamma$\textit{-iterable} if $\card{\M}< \l$ and $\M$ has a $\lambda$-iteration strategy  (or $(\a, \lambda)$-iteration strategy for some $\a\leq \l$) $\Lambda$ such that $Code_\l(\Lambda)\in\Gamma$\footnote{Recall that iteration strategy for a $\Sigma$-mouse must respect $\Sigma$. In particular, all $\Lambda$-iterates of $\M$ are $\Sigma$-premice.}. We let $Mice^{\Gamma, \Sigma}$ be the set of $\Sigma$-premice that are $\Gamma$-iterable.

\begin{definition}\label{countably iterability} Given a $\Sigma$-premouse $\M\in H_\l$, we say $\M$ is countably $\a$-iterable if whenever $\pi:\N\rightarrow \M$ is a countable submodel of $\M$, $\N$, as a $\Sigma^\pi$-mouse, is $\a$-iterable. When $\a=\omega_1+1$ then we just say that $\M$ is countably iterable. We say $\M$ is countably $\Gamma$-iterable if whenever $\pi$ and $\N$ are as above, $\N$ is $\Gamma$-iterable. 
\end{definition}

Suppose $\M$ is a $\Sigma$-premouse. We  then let $o(\M)=Ord\cap \M$. We also let $\M||\xi$ be $\M$ cutoff at $\xi$, i.e., we keep the predicate indexed at $\xi$. We let $\M|\xi$ be $\M||\xi$ without the last predicate. We say $\xi$ is a \textit{cutpoint} of $\M$ if there is no extender $E$ on $\M$ such that $\xi\in(\cp(E), lh(E)]$. We say $\xi$ is a \textit{strong cutpoint} if there is no $E$ on $\M$ such that $\xi\in[\cp(E), lh(E)]$. We say $\eta<o(\M)$ is \textit{overlapped} in $\M$ if $\eta$ isn't a cutpoint of $\M$. Given $\eta<o(\M)$ we let
\begin{center}
$\mathcal{O}^\M_\eta= \cup\{ \N\inseg\M: \rho(\N)=\eta $ and $\eta$ is not overlapped in $\N\}$. 
\end{center}

Given an swo\footnote{I.e., self well-ordered, a set $a$ is called self well-ordered if $trc(a\cup \{a\})$ is well-ordered in $L_1(a)$.} $a\in H_{\l}$ we define the stacks over $a$ by 
\begin{definition}\label{stacks} 
\begin{enumerate}
\item $Lp^{\Sigma}(a)=\cup \{ \N : \N$ is a countably iterable sound $\Sigma$-mouse over $a$ such that $\rho(\N)=a\}$,
\item $\K^{\l, \Gamma, \Sigma}(a)=\cup\{ \N: \N$ is a countably $\Gamma$-iterable sound $\Sigma$-mouse over $a$ such that $\rho(\N)=a\}$,
\item $\W^{\l,\Gamma, \Sigma}(a)=\cup \{ \N: \N$ is a $\Gamma$-iterable sound $\Sigma$-mouse over $a$ such that $\rho(\N)=a\}$.
\end{enumerate}
\end{definition}
When $\Gamma=\powerset(\cup_{\k<\l}\powerset(\k))$ then we omit it from our notation. We can define the sequences $\la Lp_\xi^{\Sigma}(a) : \xi<\eta\ra$, $\la \K_\xi^{\l, \Gamma, \Sigma}(a) : \xi<\nu\ra$, and $\la \W_\xi^{\l, \Gamma, \Sigma}(a) : \xi<\mu\ra$ as usual. For $Lp$ operator the definition is as follows:
\begin{enumerate}
\item $Lp^\Sigma_0(a)=Lp^\Sigma(a)$,
\item for $\xi<\eta$, if $Lp^{\Sigma}_\xi(a)\in H_{\l}$ then $Lp^\Sigma_{\xi+1}=Lp^{\Sigma}(Lp^\Sigma_\xi(a))$,
\item for limit $\xi<\eta$, $Lp^\Sigma_{\xi}=\cup_{\a<\xi}Lp^\Sigma_\a(a)$,
\item $\eta$ is least such that for all $\xi<\eta$, $Lp_{\xi}^{\Sigma}(a)$ is defined.
\end{enumerate}
The other stacks are similar.

\subsection{$(\Gamma, \Sigma)$-suitable premice}

Again we fix an uncountable cardinal $\l$ such that $ZF$ holds. We also fix $\Sigma\in Code_\l$ such that $\Sigma$ is a $(\a, \l)$-iteration strategy with hull condensation and $\Gamma\subseteq \powerset(\cup_{\k<\l}\powerset(\k))$ such that $Code_\l(\Sigma)\in \Gamma$. We now start outlining how to import the material from Subsection 1.3 of \cite{StrengthPFA1}. The most important notion we need from that subsection is that of $(\Gamma, \Sigma)$-suitable premouse which is defined as follows: 

\begin{definition}[$(\Gamma, \Sigma)$-suitable premouse] \label{suitable premouse}
 A $\Sigma$-premouse $\P$ is $(\Gamma, \Sigma)$-suitable if there is a unique cardinal $\delta$ such that
\begin{enumerate}
 \item $\P\models ``\d$ is the unique Woodin cardinal",
 \item $o(\P)=\sup_{n<\omega} (\d^{+n})^\P$,
 \item for every $\eta\not =\d$, $\W^{ \l,\Gamma, \Sigma}(\P|\eta)\models ``\eta$ isn't Woodin".
 \item for any $\eta<o(\P)$, $\mathcal{O}_\eta^\P=\W^{\l, \Gamma, \Sigma}(\P|\eta)$.
 \end{enumerate}
\end{definition}

Suppose $\P$ is $\Gamma$-suitable.  Then we let $\d^\P$ be the $\d$ of \rdef{suitable premouse}. We then proceed as in Section 1.3 of \cite{StrengthPFA1} to define (1) nice iteration tree, (2) $(\Gamma, \Sigma)$-short tree, (3) $(\Gamma, \Sigma)$-maximal tree, (4) $(\Gamma, \Sigma)$-correctly guided finite stack and (5) the last model of a $(\Gamma, \Sigma)$-correctly guided finite stack by using $\W^{\l, \Gamma, \Sigma}$ operator instead of $\W^\Gamma$ operator. Next, we let 

\begin{definition}[$S(\Gamma, \Sigma)$ and $F(\Gamma, \Sigma)$] $S(\Gamma, \Sigma)=\{ \Q: \Q$ is $(\Gamma, \Sigma)$-suitable$\}$. Also, we let $F(\Gamma, \Sigma)$ be the set of functions $f$ such that $dom(f)=S(\Gamma, \Sigma)$ and for each $\P\in S(\Gamma, \Sigma)$, $f(\P)\subseteq \P$ and $f(\P)$ is amenable to $\P$, i.e., for every $X\in \P$, $X\cap f(\P)\in \P$.
\end{definition}

 Given $\P\in S(\Gamma, \Sigma)$ and $f\in F(\Gamma, \Sigma)$ we let $f_n(\P)=f(\P)\cap \P|((\d^\P)^{+n})^\P$. Then $f(\P)=\cup_{n<\omega}f_n(\P)$. We also let 
\begin{center}
$\gg^\P_{f}=\d^\P \cap Hull^{\P}_1( \{ f_n(\P) : n<\omega \}) $.
\end{center}
Notice that
\begin{center}
$\gg^\P_{f}=\d^\P \cap Hull^{\P}_1(\gg_{f}^\P\cup\{ f_n(\P) : n<\omega \} )$.
\end{center}
We then let 
\begin{center}
$H_{f}^\P =Hull^{\P}_1(\gg^\P_{f}\cup \{ f_n(\P) : n<\omega \} )$.
\end{center}
If $\P\in S(\Gamma, \Sigma)$, $f\in F(\Gamma, \Sigma)$ and $i: \P\rightarrow \Q$ is an embedding then we let $i(f(\P))=\cup_{n<\omega}i(f_n(\P))$.

The following are the next block of definitions that routinely generalize into our context: (1) $(f, \Sigma)$-iterability,
(2) $\vec{b}=\la b_k: k< m\ra$ witness $(f, \Sigma)$-iterability for $\VT=\la \T_k, \P_k : k< m\ra$, and
(3) strong $(f, \Sigma)$-iterability.
These definitions generalize by using $S(\Gamma, \Sigma)$ and $f\in F(\Gamma, \Sigma)$ instead of $S(\Gamma)$ and $F(\Gamma)$.  

If $\P$ is strongly $(f, \Sigma)$-iterable and $\VT$ is a $(\Gamma, \Sigma)$-correctly guided finite stack on $\P$ with last
model $\R$ then we let
\begin{center}
$\pi^{\Sigma}_{\P, \R, f}:H_f^\P\rightarrow H_f^\R$
\end{center}
be the embedding given by any $\vec{b}$ which witnesses the $(f, \Sigma)$-iterability of $\VT$, i.e., fixing
$\vec{b}$ which witnesses $f$-iterability for $\VT$,
\begin{center}
$\pi^\Sigma_{\P, \R, f} =\pi_{\VT, \vec{b}}\rest H_f^\P$.
\end{center}
Clearly, $\pi^{\Sigma}_{\P, \R, f}$ is independent of $\VT$ and $\vec{b}$. Here we keep $\Sigma$ in our notation for $\pi^{\Sigma}_{\P, \R, f}$ because it depends on a $(\Gamma, \Sigma)$-correct iterations. It is conceivable that $\R$ might also be a $(\Gamma, \Lambda)$-correct iterate of $\P$ for another $\Lambda$, in which case $\pi^{\Sigma}_{\P, \R, f}$ might be different from $\pi^{\Lambda}_{\P, \R, f}$. However, the point is that these embeddings agree on $H_f^\P$. Also, we do not carry $\Gamma$ in our notation as it is usually understood from the context.

Given a finite sequence of functions $\vec{f}=\la f_i : i<n\ra \in F(\Gamma, \Sigma)$, we let $\oplus_{i<n}f_i\in F(\Gamma, \Sigma)$ be the function given by $(\oplus_{i<n}f_i)(\P)=\la f_i(\P): i<n\ra$. We set $\oplus\vec{f}= \oplus_{i<n}f_i$. 

We then let
\begin{center}
$\mathcal{I}_{\Gamma, F, \Sigma}=\{ (\P, \vec{f}): \P\in S(\Gamma, \Sigma)$, $\vec{f}\in F^{<\omega}$ and $\P$ is strongly $\oplus\vec{f}$-iterable$\}$.
\end{center}

\begin{definition}
Given $F\subseteq F(\Gamma, \Sigma)$, we say $F$ is \textit{closed} if for any $\vec{f}\subseteq  F^{<\omega}$ there is $\P$ such that $(\P, \oplus\vec{f})\in \mathcal{I}_{\Gamma, F, \Sigma}$ and for any $\vec{g}\subseteq F^{<\omega}$, there is a $(\Gamma, \Sigma)$-correct iterate $\Q$ of $\P$ such that $(\Q, \vec{f}\cup\vec{g})\in \mathcal{I}_{\Gamma, F, \Sigma}$.
\end{definition}  

Fix now a closed $F\subseteq F(\Gamma, \Sigma)$. Let 
\begin{center}
$\mathcal{F}_{\Gamma, F, \Sigma}=\{ H^\P_f: (\P, f)\in \mathcal{I}_{\Gamma, F, \Sigma}\}$.
\end{center}
We then define $\preceq_{\Gamma, F, \Sigma}$ on $\mathcal{I}_{\Gamma, F, \Sigma}$ by letting $(\P, \vec{f})\preceq_{\Gamma, F, \Sigma} (\Q, \vec{g})$ iff $\Q$ is a $(\Gamma, \Sigma)$-correct iterate of $\P$ and $\vec{f}\subseteq \vec{g}$. Given $(\P, \vec{f})\preceq_{\Gamma, F, \Sigma} (\Q, \vec{g})$, we have that 
\begin{center}
$\pi^\Sigma_{\P, \Q, \vec{f}}: H^\P_{\oplus\vec{f}}\rightarrow H^\Q_{\oplus\vec{f}}$.
\end{center}
Notice that if $F$ is closed then $\preceq_{\Gamma, F, \Sigma}$ is directed. Let then 
\begin{center}
$\M_{\infty, \Gamma, F, \Sigma}$
\end{center}
be the direct limit of $(\mathcal{F}_{\Gamma, F, \Sigma}, \preceq_{\Gamma, F, \Sigma})$ under $\pi^\Sigma_{\P, \Q, \vec{f}}$'s. Given $(\P, \vec{f})\in \mathcal{I}_{\Gamma, F, \Sigma}$, we let $\pi^\Sigma_{\P, \vec{f}, \infty}: H^\P_{\oplus\vec{f}}\rightarrow \M_{\infty, \Gamma, F, \Sigma}$ be the direct limit embedding. Using the proof of Lemma 1.19 of \cite{StrengthPFA1}, we get that

\begin{lemma} $\M_{\infty, \Gamma, F, \Sigma}$ is wellfounded.
\end{lemma}

The following list is then the next block of definitions that carry over to our context with no significant changes:
(1) semi $(F, G, \Sigma)$-quasi iteration,
(2) the embeddings of the $(F, G, \Sigma)$-quasi iteration (in this context, we will have $\Sigma$ in the superscripts),
(3) $(F, G, \Sigma)$-quasi iterations,
(4) the last model of $(F, G)$-quasi iterations,
(5) $\vec{f}$-guided strategies,
(6) a $\Sigma$-quasi-self-justifying-system ($\Sigma$-qsjs) and
(7) $(\omega, \Gamma, \Sigma)$-suitable premice.

\subsection{$\H_\Sigma$ under $AD^+$}
It turns out that for certain iteration strategies $\Sigma$, $V_\Theta^{\H_{\Sigma}}$ of many models of determinacy can be obtained as $\M_{\infty, \Gamma, F, \Sigma}$ for some $\Gamma$ and $F$. For the rest of this section we assume $AD^+$. Suppose $\Sigma$ is an iteration strategy of some hod mouse $\Q$ and suppose $\Sigma$ is $\powerset(\powerset(\omega))$-fullness preserving (see \cite{ATHM}) and has branch condensation (i.e., we take $\l=\omega_1$). Assume further that  $V=L(\powerset(\mathbb{R}))+MC(\Sigma)+\Theta=\theta_\Sigma$ and that $\P$ is below ``$\theta$ is measurable", i.e., below measurable limit of Woodins. We let $\Gamma=\powerset(\powerset(\omega))$ and for the duration of this subsection, we drop $\Gamma$ from our notation. Thus, a $\Sigma$-suitable premouse is a $(\Gamma, \Sigma)$-suitable premouse and etc.

Suppose $\P$ is $\Sigma$-suitable and $A\subseteq \bR$ is $OD_\Sigma$. We say $\P$ \textit{weakly term captures} $A$ if letting $\d=\d^\P$, for each $n<\omega$ there is a term relation $\tau\in \P^{Coll(\omega, (\d^{+n})^\P)}$ such that for comeager many $\P$-generics, $g\subseteq Coll(\omega, (\d^{+n})^\P)$, $\tau_g=\P[g]\cap A$. We say $\P$ \text{term captures} $A$ if the equality holds for all generics. The following lemma is essentially due to Woodin and the proof for mice can be found in \cite{CMI}.

\begin{lemma} Suppose $\P$ is $\Sigma$-suitable and $A\subseteq \bR$ is $OD_\Sigma$. Then $\P$ weakly term captures $A$. Moreover, there is a $\Sigma$-suitable $\Q$ which term captures $A$.
\end{lemma}

Given a $\Sigma$-suitable $\P$ and an $OD_\Sigma$ set of reals $A$, we let $\tau_{A, n}^\P$ be the standard name for a set of reals in $\P^{Coll(\omega, (\d^{+n})^\P)}$ witnessing the fact that $\P$ weakly captures $A$. We then define $f_A\in F(\Gamma, \Sigma)$ by letting
\begin{center}
$f_A(\P)=\la \tau^\P_{A, n} : n<\omega\ra$.
\end{center}
Let $F_{\Sigma, od}=\{ f_A: A\subseteq \bR \wedge A\in OD_\Sigma\}$. 

 All the notions we have defined above using $f\in F(\Gamma, \Sigma)$ can be redefined for $OD_\Sigma$ sets $A\subseteq \mathbb{R}$ using $f_A$ as the relevant function. To save some ink, in what follows, we will say $A$-iterable instead of $f_A$-iterable and similarly for other notions. Also, we will use $A$ in our subscripts instead of $f_A$.

 The following lemma is one of the most fundamental lemmas used to compute $\H$ and it is originally due to Woodin. Again, the proof can be found in \cite{CMI}.

\begin{theorem}\label{existence of quasi-iterable premice} For each $f\in F_{\Sigma, od}$, there is $\P\in S(\Gamma, \Sigma)$ which is $(F_{\Sigma, od}, f)$-quasi iterable.
\end{theorem}

Let $\M_\infty=\M_{\infty, F_{od}, \Sigma}$. 

\begin{theorem}[Woodin, \cite{CMI}]\label{hod theorem} $\d^{\M_\infty}=\Theta$, $\M_\infty\in \H_\Sigma$ and 
\begin{center}
$\M_\infty|\Theta=(V_\Theta^{\H_\Sigma}, \vec{E}^{\M_\infty|\Theta}, S^{\M_\infty}, \in)$
\end{center}
 where $S^{\M_\infty}$ is the predicate of $\M_\infty$ describing $\Sigma$.
\end{theorem}

Finally, if $a\in H_{\omega_1}$ is an swo then we could define $\M_\infty(a)$ by working with $\Sigma$-suitable premice over $a$. Everything we have said about $\Sigma$-suitable premice can also be said about $\Sigma$-suitable premice over $a$ and in particular, the equivalent of \rthm{hod theorem} can be proven using $\H_{(\Sigma, a)\cup\{a\}}$ instead of $\H_\Sigma$ and $\M_\infty(a)$ instead of $\M_\infty$. 

\section{The core model induction} 

The core model induction is a method for constructing models of determinacy while working under various hypothesis. The goal of this section is to develop some basic notions in order to state \rthm{the cmi theorem} which we will use as a black box. Our core model induction is a typical one: we have two uncountable cardinals $\k<\l$, \textit{the core model induction operators} (cmi operators) defined on bounded subsets of $\kappa$ can be extended to act on bounded subsets of $\l$, and for any such cmi operator $F$ acting on bounded subsets of $\l$, the minimal $F$-closed mouse with one Woodin cardinal exists and is $\l$-iterable. Having these three conditions is enough to show, by using the \textit{scales analysis} developed in \cite{K(R)} and \cite{NamScales}, that the \textit{maximal model of $AD^+$} at $\k$ satisfies $AD^+$. The details of the proof of \rthm{the cmi theorem} have appeared, in a less general form, in \cite{CMI} and \cite{PFA}. 

We start by introducing \textit{extendable strategies} and \textit{mouse operators}. We assume $ZFC$ and fix an uncountable cardinal $\l$. 

\begin{definition}[Extendable operators] Suppose $\Lambda\in dom(Code_\l)$ and $\nu\geq \l$. We say $\Lambda$ is $\nu$-extendable if whenever $g\subseteq Coll(\omega, <\l)$, there is a unique $\Psi\in dom(Code_\nu^{V[g]})$ such that $\Psi\rest H_{\l}^V=\Lambda\rest H_{\l}^V$. We also say $\Psi$ is the extension of $\Lambda$ in $V[g]$ and write $\Lambda^g$ for $\Psi$. If $\Lambda$ is $\nu$-extendable for all $\nu<\a$ then we say $\Lambda$ is $<\a$-extendable. When $\a=Ord$ we drop it from our notation and say $\Lambda$ is extendable.
\end{definition}

Suppose now $g$ is a $<\l$-generic, $a\in (H_{\l})^V[g]$ and $\Lambda\in dom(Code_\l)$ is $\l$-extendable.  Then we define $Lp^{\Sigma, g}(a)$, $\W^{\l, \Sigma, g}(a)$ and $\K^{\l, \Sigma, g}(a)$  in $V[g]$ according to \rdef{stacks}. The following connects the three stacks defined above.

\begin{proposition}\label{facts on stacks} For every $a\in H^V_{\l}$, $\W^{\l, \Sigma}(a)\insegeq \K^{\l,\Sigma}(a)\insegeq Lp^\Sigma(a)$. Suppose further that $\eta<\l$, $g\subseteq Coll(\omega, \eta)$ or $g\subseteq Coll(\omega,<\eta)$ is $V$-generic and $\Sigma$ has a unique extension $\Sigma^g$ in $V[g]$. Then $\W^{\l, \Sigma, g}(a)\insegeq \W^{\l, \Sigma}(a)$, $\K^{\l, \Sigma, g}(a)\insegeq \K^{\l, \Sigma}(a)$ and $Lp^{\Sigma, g}(a)\insegeq Lp^{\Sigma}(a)$.
\end{proposition}

We are now in a position to introduce the maximal model of $AD^+$.

\begin{definition}[Maximal model of $AD^+$]\label{the max model}
Suppose $\Sigma\in Code_\l$ is $\l$-extendable and $\mu\leq \l$ is a cardinal such that $\M_\Sigma\in H_\mu$. Let $g\subseteq Coll(\omega, <\mu)$ be generic. Then we let $\S_{\mu, g}^{\l, \Sigma}=L(\K^{\l, \Sigma, g}(\mathbb{R}^{V[g]}))$.
\end{definition}
Thus far strategy mice have been discussed only in situations when the underlying set was an swo. However, $\S_{\mu, g}^{\l, \Sigma}$ is a $\Sigma$-mouse over the set of reals. Such hybrid mice were defined in Section 2.10 of \cite{ATHM}. We say that $\S_{\mu, g}^{\l, \Sigma}$ is the $\l$-maximal models of $AD^+$ at $\mu$. Suppose now that $(\P, \Sigma)$ is a hod pair\footnote{Hod pairs are in the sense of \cite{ATHM}. They all satisfy that there is no measurable limit of Woodins.} such that $\Sigma\in dom(Code_\l)$, $\P \in H_\l$, $\Sigma$ has branch condensation and $\Sigma$ is $\l$-extendable. Then we let $H(\l, \Sigma)$ stand for the following statement:\\

$H({\l}, \Sigma)$: There is some $\a$ such that whenever $g\subseteq Coll(\omega, <\l)$ is $V$-generic, in $V[g]$, $L_\a^{\Sigma^g}(\bR)\models AD^++SMC$ and $\Sigma^g\rest HC^{V[g]}$ is $(\powerset(\mathbb{R}))^{L_\a^{\Sigma^g}(\bR)}$-fullness preserving. \\


We are now in a position to define hod pairs below a cardinal. 

\begin{definition}[Hod pair below $\l$]\label{hod pair below lambda} Suppose $(\P, \Sigma)$ is as above. Then  we say $(\P, \Sigma)$ is a hod pair below $\l$ if $\P\in H_\l$ and $H(\l, \Sigma)$ holds.
\end{definition}

The mouse operators that are constructed during core model induction have two additional properties: they \textit{transfer and relativize well}. To make this notions precise, fix $\Sigma\in dom(Code_\l)$ which is $\l$-extendable. Given a $\Sigma$-mouse operator $F\in dom(Code_\l)$, we say 
\begin{enumerate}
\item (Relativizes well) $F$ relativizes well if there is a formula $\phi(u, v, w)$ such that whenever $X, Y\in dom(F)$ and $N$ are such that $X\in L_1(Y)$ and $N$ is a transitive rudimentary closed set such that $Y, F(Y)\in N$ then $F(X)\in N$ and $F(X)$ is the unique $U$ such that $N\models \phi[U, X, F(Y)]$.
\item (Transfers well) $F$ transfers well if whenever $X, Y\in dom(F)$ are such that $X$ is generic over $L_1(Y)$ then $F(L_1(Y)[X])$ is obtained from $F(Y)$ via $S$-constructions (see Section 2.11 of \cite{ATHM}) and in particular, $F(L_1(Y))[X]=F(L_1(Y)[X])$. 
\end{enumerate}

We are now in a position to introduce the core model induction operators that we will need in this paper. 

\begin{definition}[Core model induction operator]\label{cmi operator} Suppose $\card{\bR}=\k$, $(\P, \Sigma)$ is a hod pair below $\k^+$ such that $\card{\P}<\k$, $a\in HC$, $\M\insegeq \mathcal{W}^{\k^+}(a)$ such that $\rho(\M)=a$ and $\Lambda$ is $\M$'s unique strategy. We say $F\in dom(Code_{\k^+})$ is a $(\Sigma, \Lambda)$ core model induction operator or just $(\Sigma, \Lambda)$-cmi operator if one of the following holds: For some $\a\in Ord$
\begin{enumerate}
\item  letting $M=\S^{\k^+, \Lambda}_{\omega}||\a$, $M\models AD^++MC(\Sigma)$\footnote{$MC(\Lambda)$ stands for the Mouse Capturing relative to $\Lambda$ which says that for $x, y\in \bR$, $x$ is $OD(\Lambda, y)$ iff $x$ is in some $\Lambda$-mouse over $y$.} and one of the following holds:
\begin{enumerate}
 \item $F$ is a $\Sigma$-mouse operator which transfers and relativizes well. 
\item For some swo $b\in HC$ and some $\Sigma$-premouse $\Q\in HC^{V}$ over $b$, $F$ is an $(\omega_1, \omega_1)$-iteration strategy for $\Q$ which is $(\powerset(\mathbb{R}))^{M}$-fullness preserving and $\a$ ends either a weak or a strong gap in the sense of \cite{NamScales}. 
\item  For some $H\in dom(Code_{\k^+})$, $H$ satisfies a or b above and for some $n<\omega$, $F$ is $x\rightarrow \M^{\#, H}_n(x)$ operator or for some $b\in HC$, $F$ is the $(\k^+, \k^+)$-iteration strategy of $\M_n^{\#, H}(b)$.
\end{enumerate}
\item The above conditions hold for $F$ with $L^\Lambda_{\k^+}(\bR)$ used instead of $\S^{\k^+, \Lambda}_{\omega}$ and $\Lambda$ used instead of $\Sigma$.
\end{enumerate}
We say $F$ is a $\Sigma$-cmi operator if for some $\Lambda$, $F$ is a $(\Sigma, \Lambda)$-cmi operator.
\end{definition}

When $\Sigma=\emptyset$ then we omit it from our notation. Often times, when doing core model induction, we have two uncountable cardinals $\k<\l$ and we need to show that cmi operators in $dom(Code_\k)$ are $\l$-extendable. We also need to know that given any $\l$-cmi operator $F$, $\M_1^{\#, F}$-exists. We make these statements more precise.

\begin{definition}[Lifting cmi operators]\label{lifting cmi operators} Suppose $\k<\l$ are two cardinals such that $\k$ is an inaccessible cardinal and suppose $(\P, \Sigma)$ is a hod pair below $\k$.\begin{enumerate}
\item Lift$(\k, \l, \Sigma)$ is the statement that $\Sigma$ is $\l$-extendable and for every generic $g\subseteq Coll(\omega, <\k)$, in $V[g]$, every $\Sigma^g$-cmi operator $F$ is $\l$-extendable. If Lift$(\k, \l, \Sigma)$ holds, $g\subseteq Coll(\omega, <\k)$ is generic, and $F$ is a $\Sigma$-cmi operator $F$ then we let $F^\l$ be its lifted version.
\item We let Proj$(\k, \l, \Sigma)$ be the conjunction of the following statements: for every generic $g\subseteq Coll(\omega, <\k)$, in $V[g]$,
\begin{enumerate}
\item  for every $\Sigma^g$-cmi operator $F$ which is $\l$-extendable, $\M_1^{\#, F}$ exists and is $\l$-iterable via a $\l$-extendable strategy. 
\item for every $a\in H_{\omega_2}$, $\mathcal{K}^{\omega_1, \Sigma, g}(a)=\mathcal{\W}^{\l, \Sigma, g}(a)$
\end{enumerate}
\end{enumerate}
\end{definition}

Recall that under $AD$, if $X$ is any set then $\theta_X$ is the least ordinal which isn't a surjective image of $\bR$ via an $OD_X$ function. 

\begin{theorem}\label{the cmi theorem} Suppose $\k<\l$ are two uncountable cardinals and suppose $(\P, \Sigma)$ is a hod pair below $\k$ such that Lift$(\k, \l, \Sigma)$ and Proj$(\k, \l, \Sigma)$ hold. Then for every generic  $g\subseteq Coll(\omega, <\k)$, one of the following holds: 
\begin{enumerate}
\item  $\mathcal{S}_{\k, g}^{\l, \Sigma}\models AD^++\theta_\Sigma=\Theta$.
\item There is $A\subseteq \bR$ such that $\Sigma^g\in L(A, \bR)$, $L(A, \bR)\models AD^++MC(\Sigma^g)+\theta_{\Sigma^g}<\Theta$ and $(Lp^\Sigma(\bR))^{L(A, \bR)}\insegeq \S_{\k, g}^{\l, \Sigma}$.
\end{enumerate}
\end{theorem}

The proof of the theorem is very much the proof of the core model induction theorems in \cite{StrengthPFA1} (see Theorem 2.4 and Theorem 2.6), \cite{CMI} (see Chapter 7) and \cite{PFA}. Since there are no new ideas in the proof of \rthm{the cmi theorem} we omit the proof. One remark is that under the hypothesis of \rthm{the cmi theorem},  whenever $\Lambda\in V[g]$ is an iteration strategy of some $\Sigma$-mouse $\M$ over some swo $a\in HC^{V[g]}$ with the property that $\rho(\M)=a$ then $L^\Lambda(\bR^{V[g]})\models AD^+$. It then follows that if clause 2 fails then $\mathcal{S}_{\k, g}^{\l, \Sigma}\models ``V=OD_{\Sigma, \bR}$" and in particular, $\mathcal{S}_{\k, g}^{\l, \Sigma}\models \theta_\Sigma=\Theta$. Hence, we could have omitted it from clause 1. 

The following is a useful fact on lifting strategies.

\begin{lemma}[Lifting cmi operators through strongness embeddings]\label{lifting cmi operators through strongness embeddings}
Suppose $\k<\l$ are such that $\k$ is a $\l$-strong cardinal. Then whenever $(\P, \Sigma)$ is a hod pair below $\k$ which is $\k$-extendable then Lift$(\k, \l, \Sigma)$ and clause b of Proj$(\k, \l, \Sigma)$ hold. 
\end{lemma}
\begin{proof}
Fix an embedding $j:V\rightarrow M$ witnessing $\k$ is $\l$-strong. We only show that Lift$(\k, \l, \Sigma)$ holds as the proof of clause b of Proj$(\k, \l, \Sigma)$ is very similar. Let $g\subseteq Coll(\omega, <\k)$ and $h\subseteq Coll(\omega, <j(\k))$ be $V$-generic such that $g=h\cap Coll(\omega, <\k)$. We can then extend $j$ to $j^+: V[g]\rightarrow M[h]$.

Working in $V[g]$, fix $F$ which is a $\Sigma^g$-cmi operator. We want to show that $F$ is $\l$-extendable. Let then $F^+=j^+(F)\rest H_\l[g]$. Because for each $\a<\b<\k$ such that $\M_F\in H_\a$, $F^g\rest H_\b^{V[g\cap Coll(\omega, <\a)]}\in V[g\cap Coll(\omega, <\a)]$ we have that $F^+\in V[g]$ and it extends $F$. 

Next, we need to see that there is a unique such $F^+$. Suppose then $H\in dom(Code_\l^{V[g]})$ is another extension of $F$. Because $\k=\omega_1^{V[g]}$, a simple Skolem hull argument gives a contradiction. Indeed, working in $V[g]$, let $\pi: N\rightarrow H_{\l^+}[g]$ be an elementary such that $N$ is countable and $F^+, H\in rng(\pi)$. Let $(\bar{F}, \bar{H})=\pi^{-1}(F^+, H)$. Then it follows from the definition of being a $\Sigma$-cmi operator that $\bar{F}=F^+\rest N$ and $\bar{H}=H\rest N$. However, since $F^+\rest N=F\rest N=H\rest N$, we get that $N\models \bar{F}=\bar{H}$, contradiction! 
\end{proof}

\section{A core model induction at a strong cardinal}

In this section we present a useful application of \rthm{the cmi theorem} which we will later use to prove our main theorem. 

\begin{theorem}\label{main technical theorem} Suppose $\mu<\k<\l$ are such that $\l$ is an inaccessible cardinal, $\mu$ and $\k$ are $\l$-strong and whenever $(\P, \Sigma)$ is a hod pair below $\k$ such that $\l^\P=0$, Proj$(\k, \l, \Sigma)$ holds. Suppose $g\subseteq Coll(\omega,<\k)$ is generic.  Then in $V[g]$, there is $A\subseteq \bR$ such that $L(A, \bR)\models \theta_0<\Theta$.
\end{theorem}

We present the proof of \rthm{main technical theorem} in a sequence of lemmas. Fix then $\mu<\kappa<\l$ as in \rthm{main technical theorem}. Towards a contradiction we assume that 
\begin{center}
(*)\ \ \ \ \ \ \ for any generic $h\subseteq Coll(\omega, <\k)$, in $V[h]$, there is no $A\subseteq \bR$ such that $L(A, \bR)\models \theta_0<\Theta$.
\end{center}

Fix a $V$-generic $g\subseteq Coll(\omega, <\mu)$ and let $E$ be a $(\mu, \l)$-extender such that $\cp(E)=\mu$ and $V_{\l}\subseteq Ult(V, E)$. We let $j=j_E$ and $M=Ult(V, E)$. Also fix a $V$-generic $h\subseteq Coll(\omega, <j(\mu))$ such that $h\cap Coll(\omega, <\mu)=g$. It then follows that $j$ lifts to $j^+: V[g]\rightarrow M[h]$. For the rest of this section we let $W=V[g]$. The next lemma shows that the various models that we have defined compute the stacks we have defined correctly. Below, if $\xi\in Ord$ and $N$ is a transitive model of $ZFC$ then we let $N_\xi=V_\xi^N$. Among other things the next lemma can be used to show clause b of Proj$(\k, \l, \Sigma)$.

\begin{lemma}\label{correctness facts} Suppose $(\P, \Sigma)$ is a hod pair below $\mu$ and $a\in W_\l$ is an swo. Then
\begin{center}
$\W^{\l, \Sigma, g}(a)=\K^{\l, \Sigma, g}(a)=(Lp^\Sigma(a))^W=(\W^{j(\l), j(\Sigma), h}(a))^{M[h]}$.
\end{center}
\end{lemma}
\begin{proof}
It is enough to show that $\W^{\l, \Sigma, g}(a)=(Lp^\Sigma(a))^W$ and $\W^{\l, \Sigma, g}(a)=(\W^{j(\l), j(\Sigma), h}(a))^{M[h]}$. We start with the first. Work in $W$. Clearly $\W^{\l, \Sigma, g}(a)\insegeq (Lp^\Sigma(a))^W$. Let then $\M\insegeq (Lp^\Sigma(a))^W$ be such that $\rho(\M)=a$. We want to see that $\M\insegeq \W^{\l, \Sigma, g}(a)$. To see this, notice that by a standard absoluteness argument, there is $\sigma: \M\rightarrow j^+(\M)$ such that $\sigma\in M[h]$, $\sigma(\P)=\P$ and $M[h]\models j(\Sigma^g)^\sigma=j(\Sigma^g)$ (this follows from the fact that $\Sigma$ has branch condensation). Hence, in $M[h]$, $\M$ is $\omega_1+1$-iterable $j(\Sigma^g)$-mouse. Let in $M[h]$, $\Lambda\in M[h]$ be the unique $\omega_1+1$-iteration strategy of $\M$ (as a $j(\Sigma^g)$-mouse). It follows from the homogeneity of the collapse and the uniqueness of $\Lambda$ that $\Lambda\rest H_{\l}^W\in W$. Hence, $\M\insegeq \W^{\l, \Sigma, g}(a)$. 

To see that $\W^{\l, g}(a)=(\W^{j(\l), j(\Sigma), h}(a))^{M[h]}$, first suppose $\M\insegeq \W^{\l, \Sigma, g}(a)$. Then, in $M[h]$, $j(\M)\insegeq \W^{j(\l), j(\Sigma), h}(j^+(a))$. Since, in $M[h]$, $\M$ is embeddable into $j^+(\M)$ via $\sigma$ with the above properties, we get that in $M[h]$, $\M\insegeq \W^{j(\l), j(\Sigma), h}(a)$.  Next, suppose $\M\insegeq (\W^{j(\l), j(\Sigma), h}(a))^{M[h]}$ is such that $\rho(\M)=a$. It follows from the homogeneity of the collapse and the uniqueness of the strategy of $\M$ that $\M\in W$ and that $\M\insegeq \W^{\l, \Sigma, g}(a)$.
\end{proof}

Because we are assuming (*), it follows from  \rthm{the cmi theorem} and \rlem{lifting cmi operators through strongness embeddings} that 

\begin{corollary} $\S^\l_{\mu, g}\models AD^++\Theta=\theta_0$.
\end{corollary}
\begin{proof} It follows from \rlem{lifting cmi operators through strongness embeddings} that Lift$(\mu, \l)$ holds. Suppose then that $\neg (\S^\l_{\mu, g}\models AD^++\Theta=\theta_0)$. Working in $V[g]$, using \rthm{the cmi theorem}, we can fix $A\subseteq \bR$ such that $L(A, \bR)\models \theta_{0}<\Theta$ and $(Lp(\bR))^{L(A, \bR)}\insegeq \S_{\mu, g}^{\l}$. By the results of \cite{ATHM}, there is $(\P, \Sigma)\in L(A, \bR)$ such that $\l^\P=0$ and letting $\mathcal{H}=\H^{L(A, \bR)}$, in $L(A, \bR)$, $\Sigma$ is fullness preserving, has branch condensation and 
\begin{center}
$\M_\infty|\theta_0=(V_{\theta_0}^{\mathcal{H}}, \vec{E}^{\M_\infty|\theta_0},  \in)$.
\end{center}

Using \rlem{lifting cmi operators through strongness embeddings}, we can extend $\Sigma$ to a $(\l, \l)$-strategy. Let $\Sigma^\l$ be this strategy. It follows from \rlem{correctness facts} that $\Sigma^\l$ is $\powerset(\cup_{\gamma<\lambda}\powerset(\gamma))$-fullness preserving in $V[g]$. Because Proj$(\k, \l)$-holds and because, letting $\Sigma^\k=\Sigma^\l\rest H_\k^W$, $(\P, \Sigma^\k)$ is a hod pair below $\k$, letting $h^*=h\cap Coll(\omega, <\k)$, it follows from \rthm{the cmi theorem} and (*), $\S^{\l, \Sigma^\k}_{\k, h^*}\models AD^+$. It then follows from \rlem{correctness facts} that $\S^{\l, \Sigma^\k}_{\k, h^*}\models ``\Sigma^\k$ is fullness preserving" implying that $\S^{\l, \Sigma^\k}_{\k, h^*}\models \theta_0<\Theta$. This contradicts (*).
\end{proof}

Let now $\P=(\M_\infty)^{\S^\l_{\mu, g}}$ and let in $V[g]$, $\Gamma=\{ A\subseteq \powerset(\l) :$ for some $\a\leq \l$, $Code_\l^{-1}(A)$ is a $(\a, \l)$-iteration strategy$\}$. Our ultimate goal is to produce, while working in $V[g]$, a premouse $\Q$ and a $(\lambda, \lambda)$-iteration strategy $\Sigma$ for $\Q$ such that $\Sigma$ is $\Gamma$-fullness preserving and has branch condensation. We start by describing a strategy for $\P$ which is fullness preserving. During the rest of this section, we drop $\Gamma$ from our notation. Thus, a suitable premouse is $\Gamma$-suitable premouse and etc. Let $k=h\cap Coll(\omega, <\l)$, $\S=\S^\l_{\mu, g}$ and $\Gamma^*=(F_{od})^\S$. 

\begin{lemma}\label{getting qsjs} $j^+[\Gamma^*]$ is a qsjs for $j^+(S(\Gamma^*))$ as witnessed by $\P$ . 
\end{lemma}
\begin{proof} The lemma easily follows from the following claim.\\

\textit{Claim.}  Suppose $\R\in j(\S)$ is such that there are $\pi:\P\rightarrow \R$ and $\sigma:\R\rightarrow j(\P)$ such that $j\rest \P=\sigma\circ \pi$. Then $\R\in S(j^+(\Gamma^*))$. 
\begin{proof}
First let $T\in j^+(\S)$ be the tree projecting to the universal $(\Sigma^2_1)^{j^+(\S)}$ set. We have that $L[T, \P]\models \P=H_{(\d^\P)^{+\omega}}$. Notice that $T\in V$. It then follows that we can lift $j\rest \P$, $\pi$ and $\sigma$ to 
\begin{center}
$j^*:L[T, \P]\rightarrow L[j(T), j(\P)]$, $\pi^*: L[T, \P]\rightarrow L[\pi^*(T), \R]$ and $\sigma^*: L[\pi^*(T), \R]\rightarrow L[j(T), j(\P)]$.
\end{center}
such that $j^*=\sigma^*\circ \pi^*$. The proof of Lemma 2.21 of \cite{StrengthPFA1} now shows that $\R\in j^+(S(\Gamma^*))$.
\end{proof}

To finish the proof, we need to show that for every $A\in \Gamma^*$, in $j^+(\S)$, \\

(1) $\P$ is $(j^+[\Gamma^*], j(A))$-quasi iterable and

(2) whenever $\Q$ is a $j^+[\Gamma^*]$-quasi iterate of $\P$ and $\pi:\R\rightarrow_{\Sigma_1}\Q$ is such that for every $A\in \Gamma^*$, $\tau_{j^+(A)}^\Q\in rng(\pi)$ then $\R\in j^+(\S(\Gamma^*))$.\\

We prove (1) as the proof of (2) is very similar. Fix $A\in \Gamma^*$ and fix $\Q\in S(\Gamma^*)$ such that in $\S$, $\Q$ is $(\Gamma^*, A)$-quasi iterable. Then $j^+(\S)\models ``\Q$ is $(j^+(\Gamma^*), j(A))$-quasi iterable". Since we have that $j^+(\S)\models ``\P$ is a $(j^+(\Gamma^*), j(A))$-quasi iterate of $\Q$", we have that $j^+(\S)\models ``\P$ is a $(j^+(\Gamma^*), j(A))$-quasi iterable". Repeating the argument for every $A$, we get that \\

(3) for every $A\in \Gamma^*$, $j^+(\S)\models ``\P$ is $(j^+(\Gamma^*), j(A))$-quasi iterable".\\

It follows from (3) that to finish the proof of (1) it's enough to show that \\

(4) for every $A\in \Gamma^*$, in $j^+(\S)$, every $(j^+(\Gamma^*), j^+(A))$-quasi iteration is also a $(j^+[\Gamma^*], j^+(A))$-quasi iteration. \\

To prove (4), it is enough to show that whenever $\Q$ is a $j^+(\Gamma^*)$-quasi iterate of $\P$ then $\d^\Q=\cup_{B\in j^+[\Gamma^*]} H^\Q_{\tau_B^\Q}$. Fix then $\Q$ which is a $j^+(\Gamma^*)$-quasi iterate of $\P$. Let $\pi=\cup_{B\in j^+[\Gamma^*]} \pi_{\P, \Q, B}$, $\S$ be the transitive collapse of $\cup_{B\in j^+[\Gamma^*]} H^\Q_{\tau_B^\Q}$, $\sigma:\S\rightarrow \Q$ be the uncollapse map, and $\tau=\cup_{B\in j^+(\Gamma^*)} \pi_{\P, \Q, B}$. Because $\P=\cup_{B\in j^+[\Gamma^*]} H^\P_B$, $\pi$ is total. It then follows that 
\begin{center}
$j\rest \P=\tau\circ (\sigma^{-1}\circ \pi)$.
\end{center}
The claim then implies that $\S\in S(j^+(\Gamma^*))$. This finishes the proof of (1). The proof of (2) is very similar and again the key point is that the embedding $\pi$ defined above is total. 
\end{proof}

We can use Lemma 1.29 of \cite{StrengthPFA1} to get a strategy $\Sigma^*=\Sigma^{j^+[\Gamma^*]}$. In our current situation, there is one important difference with \cite{StrengthPFA1}. In our current context, $\Sigma^*$ may not act on all trees that are in $M[h]$ as $j^+[\Gamma^*]$ isn't in $M[h]$. However, it acts on all stacks that are in $V_\l[k]$. This is simply because 
\begin{center}
$F=\{ B\cap \bR^{V[k]}: B\in j^+[\Gamma^*]\}\in V[k]$. 
\end{center}
We then let $\Sigma=\Sigma^*\rest HC^{V[k]}$. It follows from the proof of Lemma 1.29 of \cite{StrengthPFA1} and \rlem{correctness facts} that, in $V[k]$, $\Sigma$ is a $(\l, \l)$-iteration strategy which is $\powerset(\cup_{\gamma<\lambda}\powerset(\gamma))$-fullness preserving and is guided by $F$. 

Next, we show that for some $\Sigma$-iterate $\Q$ of $\P$ via some stack $\VT$ such that $\pi^{\VT}$-exists, $\Sigma_{\Q, \VT}$ has branch condensation. We follow the proof of branch condensation that first appeared in \cite{Ketchersid} and also in Chapter 7 of \cite{CMI} (see especially the proofs of Lemma 7.9.6 and Lemma 7.9.7 of \cite{CMI}). Below we summarize what we need in order to carry out the proof. First we let $m=k\cap Coll(\omega, <\k)$ and $\Lambda=\Sigma\rest HC^{V[m]}$. We will in fact show that some tail of $\Lambda$ in $V[m]$ has branch condensation. Recall that if $\Psi$ is possibly partial iteration strategy for a suitable premouse $\R$ then we say $\Psi$ has weak-condensation on its domain if whenever $\R^*$ is a $\Psi$-iterate of $\R$ such that the iteration embedding $i:\R\rightarrow \R^*$ exists and $\R^{**}$ is such that there are $\pi:\R\rightarrow \R^{**}$ and $\sigma:\R^{**}\rightarrow \R^*$ with the property that $i=\sigma\circ \pi$ then $\R^{**}$ is suitable. 

Suppose $(R, J)$ is a pair such that $R$ is a transitive set such that for some $R$-cardinal $\nu$, $R\models ``V=H_{\nu^+}+ J$ is a precipitous ideal on $\omega_1"$. We say $(R, J)$ captures $\Lambda$ if in $V[m]$, 
\begin{enumerate}
\item $(R, J)$ is countable and an iterable pair, 
\item $\P\in HC^R$, $\Lambda\rest HC^R\in R$ and letting $\Lambda^R=\Lambda\rest HC^R$, $R\models ``$no tail of $\Lambda^R$ has branch condensation",
\item whenever $\xi<\omega_1$ and $(R_\a, J_\a, G_\a, \pi_{\a, \b} : \a<\b\leq \xi)$ is some iteration of $(R, J)$ of length $\xi+1$ then $\pi_{0, \xi}(\Lambda^R)$ has weak-condensation and fullness preservation on its domain. 
\end{enumerate}

The main lemma towards showing that some tail of $\Lambda$ has branch condensation is that

\begin{lemma}\label{no bad pair} In $V[m]$, there is no $(R, J)$ which captures $\Lambda$. 
\end{lemma}
We do not give the proof of the lemma as it can be found in \cite{Ketchersid} and in Chapter 7 of \cite{CMI}. We then derive a contradiction by showing that

\begin{lemma} In $V[m]$, there is a pair $(R, J)$ which captures $\Lambda$ and $R\models ``$no tail of $\Lambda^R$ has branch condensation". 
\end{lemma}
\begin{proof}
Recall that it follows from (*), \rthm{the cmi theorem} and \rlem{lifting cmi operators through strongness embeddings} that $\S^\l_{\k, m}\models AD^++\theta_0=\Theta$. Let then $\Q=\M_\infty^{\S^\l_{\k, m}}$. We claim that $\Q=\M_\infty(\P, \Lambda)$. This follows from the fact that $j^+(\S)\models ``\Q$ is a $F_{od}$-quasi iterate of $\P$",  from (1) and (2) of \rlem{getting qsjs} and from the fact that $\Lambda$ is $F$-guided. It now follows that $V\models \card{\Q}<\k^+$.

To finish let $\pi:\P\rightarrow \Q$ be the iteration map according to $\Lambda$. We also let $T$ be the tree of the universal $({\Sigma}^2_1)^{\S^\l_{\k, m}}$-set, $\nu=((2^\k)^+)^V$ and $\mu$ be a $\k$-complete normal measure on $\k$. Working in $V[m]$, let $\sigma: R\rightarrow (H_{\nu^+})^V[m]$ be such that $R$ is countable and $\{\Lambda, \Q, \pi, T, \mu\}\in rng(\sigma)$. Let $n\in \omega$ be such that $T_n$ projects onto $\{ (x, \M): x\in \bR^{V[m]} \wedge \M\insegeq \W^{\l, m}(x)\wedge \rho(\M)=x\}$. Also let $r\in \omega$ be such that $T_r$ projects to the set of $(x, y, z)$ such that $x$ codes an swo $X$, $y$ codes an $\M\inseg \W^{\l, m}(X)$ such that $\rho(\M)=X$ and $z$ is a tree on $\M$ according to the unique iteration strategy of $\M$.

Let then $\{ \bar{\Lambda}, \bar{\Q}, \bar{\pi}, \bar{T}, \bar{\mu}\}=\sigma^{-1}(\{ \Lambda, \Q, \pi, T, \mu\})$, $\bar{R}=\sigma^{-1}((H_{\nu^+})^V)$ and $\bar{m}=\sigma^{-1}(m)$. We then have that $R=\bar{R}[\bar{m}]$. Let then $J\in R$ be the precipitous ideal on $\omega_1$ induced by $\bar{\mu}$. (see Theorem 22.33 \cite{Jech}). 

Suppose now that no tail of $\Lambda$ has branch condensation. It then follows by elementarity of $\sigma$ that $R\models ``$no tail of $\sigma^{-1}(\Lambda)$ has branch condensation". Since we already know that in $V[m]$, $(R, J)$ is countable and iterable, to finish, it remains to show that the $(R, J)$ captures $\Lambda$. 

Let then $\Lambda^R=\Lambda\rest HC^R=\sigma^{-1}(\Lambda)$, $\Q^R=\sigma^{-1}(\Q)$ and $\pi^R=\sigma^{-1}(\pi)$. Notice that by the construction of $\Lambda$ we have that whenever $\R$ is a $\Lambda$-iterate of $\P$ via $\VT$ such that the iteration embedding $\pi^\VT$-exists then $\M_\infty(\R, \Lambda_{\R, \VT})=\Q$ and letting $\pi_{\R, \Q}$ be the iteration map, $\pi=\pi_{\R, \Q}\circ \pi^\VT$. We then have that\\

(1) $R\models $``whenever $\R$ is a $\Lambda^R$-iterate of $\P$ via $\VT$ such that the iteration embedding $\pi^\VT$-exists then $\M_\infty(\R, \Lambda^R)=\Q^R$ and letting $\pi_{\R, \Q^R}$ be the iteration map, $\pi^R=\pi_{\R, \Q^R}\circ \pi^\VT$".\\

To show that $(R, J)$ captures $\Lambda$, let $(R_\a, J_\a, G_\a, \pi_{\a, \b} : \a<\b\leq \xi)$ be some iteration of $(R, J)$ of length $\xi+1$. Let $\VT\in HC^{R_\xi}$ be according to $\pi_{0, \xi}(\Lambda^{R})$ with last model $\R$ such that $\pi^{\VT}$-exists. We need to show that $\S^{\l}_{\k, m}\models ``\R$ is $\utilde{\Sigma}^2_1$-suitable". By (1), we have that there is $p:\R\rightarrow \pi_{0, \xi}(\Q)$ such that $\pi_{0, \xi}(\pi^R)=p\circ \pi^{\VT}$.

It follows from the construction of $J$ that $\pi_{0, \xi}\rest \bar{R}$ is actually an iteration of $\bar{R}$ via $\bar{\mu}$ and hence, there is $l: \pi_{0, \xi}(\bar{R})\rightarrow H_{\nu^+}^V$ such that $\sigma\rest \bar{R}=l\circ (\pi_{0, \xi} \rest \bar{R})$. It then follows that there is $q: \pi_{0, \xi}(\bar{R})\rightarrow (H_{\nu^+})^V$ such that $\sigma\rest \bar{R}=q\circ (\pi_{0, \xi}\rest \bar{R})$. We then have that $\pi=(q\rest (\pi_{0, \xi}\rest \Q^R) )\circ p \circ \pi^\VT$, implying that, by weak condensation of $\Lambda$, that $\S^{\l}_{\k, m}\models ``\R$ is $\utilde{\Sigma}^2_1$-suitable". The proof that $\pi_{0, \xi}(\Lambda^R)$ has weak branch condensation is very similar and we omit it.

It remains to show that iterations according to $\pi_{0, \xi}(\Lambda^R)$ are correctly guided. We do this only for normal trees as the general case is only notationally more complicated. To show this, we first consider the case of trees that don't have fatal drops. Notice that if $\T\in HC^{V[m]}$ is a correctly guided tree\footnote{Recall that correctly guided trees do not have fatal drops, see the paragraph before Definition 1.11 of \cite{StrengthPFA1}.} which is according to $\Lambda$ and letting $b=\Lambda(\T)$, $\Q(b, \T)$-exists then whenever $x, y\in \bR^{V[m]}$ are such that $x$ codes $\M(\T)$ and $y$ codes $\Q(b, \T)$ then $(x, y)\in p[T_n]$. We then have that\\

(2) if $\T\in HC^{R}$ is according to $\Lambda^R$, is correctly guided and letting $b=\Lambda^R(\T)$, $\Q(b, \T)$-exists then whenever $x, y\in \bR^{R}$ are such that $x$ codes $\M(\T)$ and $y$ codes $\Q(b, \T)$ then $(x, y)\in p[\bar{T}_n]$.\\

Let now $\T\in HC^{R_\xi}$ be according to $\pi_{0, \xi}(\Lambda^R)$ and such that it is correctly guided and if $b=\pi_{0, \xi}(\Lambda^R)(\T)$ then $\Q(b, \T)$-exists. Let $x, y\in \bR^{R_\xi}$ be such that $x$ codes $\M(\T)$ and $y$ codes $\Q(b, \T)$. By (2) we have that $(x, y)\in p[\pi_{0, \xi}(\bar{T}_n)]$. Keeping the above notation, we have that $(x, y)\in p[l\circ \pi_{0, \xi}(\bar{T}_n)]=p[T_n]$ implying that $\Q(b, \T)\insegeq \W^{\l, m}(\M(\T))$. 

Lastly we need to take care of trees with fatal drops. Notice that if $\T\in HC^{V[m]}$ is a tree which has a fatal drop at $(\a, \eta)$ then letting $\U$ be the tail of $\T$ after stage $\a$ on $\mathcal{O}^{\M^\T_\a}_\eta$ and letting $\M\insegeq \mathcal{O}^{\M^\T_\a}_\eta$ be the least such that $\rho(\M)=\eta$ and $U$ is a tree on $\M$ above $\eta$ then whenever $x, y, z\in \bR^{V[m]}$ are such that $x$ codes $\M_\a^\T|\eta$, $y$ codes $\M$ and $z$  codes $\U$ then $(x, y, z)\in p[T_r]$. It then follows that\\

(3)  if $\T\in HC^{R}$ is a tree which has a fatal drop at $(\a, \eta)$ then letting $\U$ be the tail of $\T$ after stage $\a$ on $\mathcal{O}^{\M^\T_\a}_\eta$ and letting $\M\insegeq \mathcal{O}^{\M^\T_\a}_\eta$ be the least such that $\rho(\M)=\eta$ and $\U$ is a tree on $\M$ above $\eta$ then whenever $x, y, z\in \bR^{R}$ are such that $x$ codes $\M_\a^\T|\eta$, $y$ codes $\M$ and $z$  codes $\U$ then $(x, y, z)\in p[\bar{T}_r]$. \\

The rest of the proof is just like the proof of the case when $\T$ doesn't have a fatal drop except we now use (3) instead of (2). 
\end{proof}

By the lemma above we can fix a $\Lambda$-iterate $\Q$ of $\P$ via some stack $\VT$ such that $\Lambda_{\Q, \VT}$ has branch condensation. Using \rthm{the cmi theorem} and \rlem{lifting cmi operators through strongness embeddings} we get that in $V[m]$, $L(\Lambda_{\Q, \VT}, \bR)\models AD^+$. Because $\Lambda$ is fullness preserving we must have that $L(\Lambda_{\Q, \VT}, \bR)\models AD^++\theta_0<\Theta$ contradicting (*). 

\section{On the strength of the failure of the UBH for tame trees}

In this section, we present the proof of our Main Theorem. For the rest of this section we assume that there is a proper class of strong cardinals. We start by introducing \textit{tame trees}. Recall that we say $\k$ reflects the set of strong cardinals if whenever $\l\geq \k$ is strong then there is an extender $E$ witnessing that $\k$ is $\l$-strong and such that $Ult(V, E)\models ``\l$ is strong".  

\begin{definition}[Tame iteration tree]\label{tame iteration tree} A normal iteration tree $\T$ on $V$ is tame if for all $\a<\b<lh(\T)$ such that $\a=pred_T \b+1$,  $\M_{\a}^\T\models ``\exists \k<\l<\cp(E_\b^\T)$ such that $\l$ is a strong cardinal and $\k$ is strong reflecting strongs". 
\end{definition}

While our proof will not need the assumption that $\k$ is strong reflecting strongs, we defined tame trees in this particular way because we believe tame failures of $UBH$ give inner models of $AD_{\mathbb{R}}+``\Theta$ is regular". The proof of this claim will appear in a future publication.  In the proof below we will only use that there are three strong cardinals below the critical points of the branch embeddings of $\T$. 

Towards a contradiction, we assume that there is a tame iteration tree $\T$ on $V$ with two cofinal well-founded branches $b$ and $c$. Without loss of generality, we assume that $lh(\T)$ is the least possible. Let $M_b=\M^\T_b$, $M_c=\M^\T_c$, $M=\M(\T)$, $\d=\d(\T)$, $\d^+_b = (\d^+)^{M_b}$, $\d^+_c = (\d^+)^{M_c}$, $\pi_b=\pi^\T_{b}$ and $\pi_c=\pi^\T_{c}$. Also, let $\k_0<\k_1<\k_2$ be the first three strong cardinals of $V$. 

Since $\T$ is tame, we have that all the extenders used in $\T$ have critical point $>\k_2$. Suppose $g\subseteq Coll(\omega, <\k_1)$ is $V$-generic. To make the notation as transparent as possible, we will confuse our iteration embeddings that act on $V$ with their extensions that act on $V[g]$. Thus, for instance, $\pi_b:V[g]\rightarrow M_b[g]$ and etc. Working in $V[g]$, fix a hod pair $(\P, \Sigma)\in V[g]$ below $\omega_2$ such that $\P\in HC^{V[g]}$ and $\l^\P=0$. The next lemma is the key lemma. 

\begin{lemma}[Key Lemma] For every hod pair $(\P, \Sigma)$ below $\k$ such that $\l^\P=0$, Proj$(\k_1, \k_2, \Sigma)$ holds. 
\end{lemma} 

Given the Key Lemma we can easily get a contradiction by using \rthm{main technical theorem}. It is then enough to show that the Key Lemma holds which is what we will do in the next few subsections. Towards the proof of the Key Lemma, we fix a hod pair $(\P, \Sigma)$ below $\k$. Since clause b of Proj$(\k_1, \k_2, \Sigma)$ follows from \rlem{correctness facts}, we will only establish clause a. 

In the course of the proof of the Key Lemma we will heavily use the following lemma proved in \cite{steel2002core}.

\begin{lemma}
\label{UBHutility}
\begin{enumerate}
\item (Woodin) The cardinality of $\powerset(\delta)\cap M_b\cap M_c$ is at most $\delta$.
\item (Steel) $\delta$ is singular or measurable in $M_b$ (and in $M_c$).
\end{enumerate}
\end{lemma}

We will only verify clause a of Proj$(\k_1, \k_2, \Sigma)$ for $\Sigma$-cmi operators defined according to clause 1 of \rdef{cmi operator} as those defined according to clause 2 of \rdef{cmi operator} can be handled in a very similar manner. Lets then fix such a $\Sigma$-cmi operator $F$. Notice that it follows from \rlem{lifting cmi operators through strongness embeddings} that for every $\xi$, both in $M_b[g]$ and in $M_c[g]$, $F$ $\xi$-extendable. We then let $F_b$ and $F_c$ be the two $Ord$-extensions of $F$ in $M_b[g]$ and $M_c[g]$ respectively. 
 
We say $F$ can be \textit{lifted} if  for any $x\in H_{\d^+}^{M_b[g]}\cap H_{\d^+}^{M_c[g]}$, $F_b(x)=F_c(x)$ and $(Lp^{F_b}(x))^{M_b[g]}$ is compatible with $(Lp^{F_c}(x))^{M_c[g]}$ (i.e., one is an initial segment of the other).

We first present a simple lemma which illustrates some of the key ideas that we will use.

\begin{lemma}\label{compatibility of lp's} Suppose $x, \M \in M_b\cap M_c$ are such that $\M$ is a sound $x$-premouse such that $\rho(\M)=x$. Then $\M\insegeq Lp^{M_b}(x)\iff \M\insegeq Lp^{M_c}(x)$.
\end{lemma}
\begin{proof}
Suppose $\N$ is a countable hull of $\M$ in $V$. Then by an absoluteness argument, $\N$ is a countable hull of $\M$ both in $M_b$ and $M_c$. Hence, the claim follows.
\end{proof}

Unfortunately, the lemma doesn't immediately generalize to $F$-mice since the absoluteness used in the proof isn't in general true. Fixing an $\N$ as in the proof which is a countable submodel of $\M\insegeq (Lp^{F_b}(x))^{M_b}$ it is still true that $\N$ can be realized as a countable hull of $(Lp^{F_b}(x))^{M_b}$ and $(Lp^{F_c}(x))^{M_c}$ in $M_b$ and $M_c$ via certain embeddings $j_b:\N\rightarrow \M$ and $j_c:\N\rightarrow \M$ in $M_b$ and $M_c$ respectively: however, it is not clear, in the case $F$ is an iteration strategy, that $F^{j_b}_b$ and $F^{j_c}_c$ (i.e., the pullbacks of $F_b$ and $F_c$) are the same strategies. In order for these two to be the same, we seem to need to use an argument from \cite{Neeman}. 

\begin{lemma}\label{lifting cmi operators: mice} $F$ can be lifted.
\end{lemma}
\begin{proof} We already know that $F$ can be extended to $F_b$ and $F_c$. It remains to show that whenever $x\in M_b\cap M_c$, $F_b(x)=F_c(x)$ and $(Lp^{F_b}(x))^{M_b}$\footnote{This means that whenever $\pi: (\N,\P^*,x^*) \rightarrow (\M,\P,x)$ is such that $\M\lhd Lp^{F_b}(x)$ and $\N$ is countable transitive, then $\N$ has a unique $\omega_1+1$ $\Lambda$-strategy where $\Lambda$ is such that whenever $\R$ is an iterate of $\N$ and $\U \in \N$ is a tree on $\P^*$ according to $\Lambda$ then $\Lambda(\U) = F(\pi\U)\in \R$.} and $(Lp^{F_c}(x))^{M_c}$ are compatible. We show the second as the first is only a special case of it. Assume towards a contradiction that $(Lp^{F_b}(x))^{M_b}$ and $(Lp^{F_c}(x))^{M_c}$ are not compatible. Let $\S_b=(Lp^{F_b}(x))^{M_b}$ and $\S_c=(Lp^{F_c}(x))^{M_c}$. Fix $\sigma: W\rightarrow V$ such that $W$ is countable, $(\T, b, c, \P, F, x)\in rng(\sigma)$ and if $(\U, d, e, \Q, G, y)= \pi^{-1}(\T, b, c, \P, F, x)$ then $\sigma[lh(\U)]$ is cofinal in $lh(\T)$. Let $\eta=\card{\P}^M$ and let $\a\in b$ is the least such that $\cp(\pi_{\a, b})>\eta$. We then have that $\a\in c$ and $\cp(\pi_{\a, c})>\eta$ (this is because $\eta\in rng(\pi_b)\cap rng(\pi_c)$). We may also assume that $\cp(\pi_{\a, b}), \cp(\pi_{\a, c}) > \sup\{lh(E_\gamma) \ | \ \gamma < \alpha\}$. It then follows that $(\eta, \a)\in rng(\sigma)$. Let $(\nu, \b)=\sigma^{-1}(\eta, \a)$. Also, let $M_d=\M_d^\U$, $M_e=\M_e^\U$, and $(G_d, G_e)=\sigma^{-1}(F_b, F_c)$. We now have that in $W$, $(Lp^{G_d}(y))^{M_d}$ isn't compatible with $(Lp^{G_e}(y))^{M_e}$.

Let $\sigma_\xi: \M_\xi^\U\rightarrow \M_\xi^{\sigma\U}$ be the copy maps. 
By our assumption on $\beta$, it follows from Lemma 2.6 of \cite{Neeman} that $\sigma_\beta \in \M_\b^{\sigma\U}[k]$ where $k$ is a generic for a posest of size smaller than the critical point of any branch embedding starting from $\M_\b^{\sigma\U}$ and there are $m: M_d\rightarrow \M_\b^{\sigma\U}$ and $n: \M_e\rightarrow \M_\b^{\sigma\U}$ such that $\sigma_\b=m\circ \pi^\U_{\b, d}$ and $\sigma_\b=n\circ \pi^\U_{\b, e}$. Let $H = \sigma_\beta(G^*)\in \M_\b^{\sigma\U}$ where $G^* = \sigma^{-1}((\pi^\mathcal{T}_{\alpha,b})^{-1}(F)) = \sigma^{-1}((\pi^\mathcal{T}_{\alpha,c})^{-1}(F)) \in \M^\U_\beta$

Let $\R_d= (Lp^{G_d}(y))^{M_d}$ and $\R_e=(Lp^{G_e}(y))^{M_e}$. Finally, let $\W_d=m(\R_d)$ and $\W_e=n(\R_e)$. Notice that $\sigma_\b\rest \Q=m\rest \Q=n\rest \Q$ and $m(G_d), n(G_e)$ both extend $H$. Working then in $\M_\b^{\sigma\U}[k]$, we can look for maps 
\begin{enumerate}
\item $p:(\Q, G^*)\rightarrow (\sigma_\b(\Q),H)$,
\item $q: (\R_d,G_d)\rightarrow (\W_d,m(G_d))$,
\item $r:(\R_e,G_e)\rightarrow (\W_e,n(G_e))$
\end{enumerate}
such that $p=q\rest \Q=r\rest \Q$ and $q(G) = r(G) = p(G^*) = H$. By absoluteness, there must be such embeddings in $\M_\b^{\sigma\U}[k]$. But now, $\R_d$ and $\R_e$ can be compared in $\M_\b^{\sigma\U}[k]$ as the both are $G^+$-iterable where $G^+$ is $p$-pullback of $H$.
\end{proof}

Next, we show that $M[g]\models ``\M_1^{\#, F}$ exists and is $<\d$-iterable". Suppose not. We then have that $V\models ``\M_1^{\#, F}$ doesn't exist or isn't $\k_2$-iterable".  Without loss of generality, assume $\d^+_b \leq \d^+_c$. By our assumption, the $F$-closed core model $K^F$ derived from a $K^{c, F}$ which is constructed using extenders with critical point $>\k_2$ exists and is 1-small\footnote{Because we are assuming that there are proper class of strong cardinals, if such a $K^{c, F}$ construction reaches a Woodin cardinal then it also reaches $\M_1^{\#, F}$. If then such a $K^{c, F}$ construction reaches $\M_1{\#, F}$ then it must be $\k_2$-iterable as countable submodels of such a $K^{c, F}$ are $\k_2$-iterable.}. We then have that $M_b \vDash o(\pi_b(K^F)) > \delta^+$. By our smallness assumption, 
\begin{center}
$\pi_b(K^F) | (\delta^+)^{M_b} \insegeq Lp^{\pi_b(F)}(K^{\pi_b(F)}|\delta)$. 
\end{center}

The following claim then gives us a contradiction. 

\textit{Claim.} $Lp^{\pi_b(F)}(K^{\pi_b(F)}|\delta) \vDash \delta$ is Woodin.
\begin{proof}
Recall that we assume $(\delta^+)^{M_b} \leq (\delta^+)^{M_c}$. If $F$ is a strategy as in 2 or in 3 of Definition \ref{cmi operator}, this follows from Lemma \ref{lifting cmi operators: mice} and the proof of Theorem 4.1 in \cite{steel2002core}. If $F$ is a first order (hybrid) mouse operator as in 1 of Definition \ref{cmi operator}, then $Lp^{\pi_b(F)}(K^{\pi_b(F)}|\delta) \in M_b\cap M_c$ and hence by Claim 3 in the proof of Theorem 4.1 in \cite{steel2002core} and Theorem 2.2 of \cite{IT}, $Lp^{\pi_b(F)}(K^{\pi_b(F)}|\delta)\vDash \delta$ is Woodin.
\end{proof}

\section{On the strength of $\neg UBH$ without strongs}

It is possible to prove a similar lower bound for $\neg UBH$ by somewhat strengthening the hypothesis yet dropping the assumption that there are proper class of strong cardinals. In this section, we state the result. Its proof is mostly due to the second author and will appear elsewhere. 
 
Given an iteration tree $\T$ of limit length and $\a<lh(\T)$, we let $\T_{\geq \a}$ be $\T$ starting from $\a$ and $\T_{\leq \a}=\T\rest \a+1$.  Similarly, we define $\T_{<\a}$ and $\T_{>\a}$.

\begin{theorem}\label{without strongs} Suppose $\T$ is a normal tree on $V$ with two wellfounded branches $b$ and $c$ such that if $\a=\sup(b\cap c)$ then $\d(\T)\in rng(\pi^\T_{\a, b})\cap rng(\pi^\T_{\a, c})$ and $\T_{\geq \a}\in M_\a^\T$. Then in some homogenous extension of $V$ there is a transitive model $M$ such that $\bR, Ord\subseteq M$ and $M\models ``AD^++\theta_0<\Theta"$. In particular, there is a non-tame mouse.
\end{theorem}

The hypothesis of \rthm{without strongs} includes, among other trees, alternating chains.

\bibliographystyle{plain}
\bibliography{UBH-1}
\end{document}